\documentclass{my_tCON2e}

\def\etal{\mbox{et al.}}

\def\soft#1{\leavevmode\setbox0=\hbox{h}\dimen7=\ht0\advance
\dimen7 by-1ex\relax\if t#1\relax\rlap{\raise.6\dimen7
\hbox{\kern.3ex\char'47}}#1\relax\else\if T#1\relax
\rlap{\raise.5\dimen7\hbox{\kern1.3ex\char'47}}#1\relax
\else\if d#1\relax\rlap{\raise.5\dimen7\hbox{\kern.9ex
\char'47}}#1\relax\else\if D#1\relax\rlap{\raise.5\dimen7
\hbox{\kern1.4ex\char'47}}#1\relax\else\if l#1\relax
\rlap{\raise.5\dimen7\hbox{\kern.4ex\char'47}}#1\relax
\else\if L#1\relax\rlap{\raise.5\dimen7\hbox{\kern.7ex
\char'47}}#1\relax\else\message{accent \string\soft
\space #1 not defined!}#1\relax\fi\fi\fi\fi\fi\fi}

\begin{document}

\markboth{A. Debbouche and D. F. M. Torres}{Approximate Controllability of Fractional
Nonlocal Delay Semilinear Systems in Hilbert Spaces}

\title{Approximate Controllability of Fractional
Nonlocal Delay Semilinear Systems in Hilbert Spaces}

\author{Amar Debbouche$^{a, b}$
and Delfim F. M. Torres$^{b}$$^{\ast}$\thanks{$^\ast$Corresponding author.
Email: delfim@ua.pt\vspace{6pt}}\\\vspace{6pt}
$^{a}${\em{Department of Mathematics, Guelma University, Guelma, Algeria}};
$^{b}${\em{CIDMA -- Center for Research and Development in Mathematics and Applications,
Department of Mathematics, University of Aveiro,
3810-193 Aveiro, Portugal}}\\\vspace{6pt}\received{Submitted 24-Sept-2012; revised 28-Feb-2013; accepted 29-Mar-2013}}

\maketitle


\begin{abstract}
We study the existence and approximate controllability
of a class of fractional nonlocal delay semilinear
differential systems in a Hilbert space. The results
are obtained by using semigroup theory, fractional calculus,
and Schauder's fixed point theorem. Multi-delay controls
and a fractional nonlocal condition are introduced.
Furthermore, we present an appropriate set of
sufficient conditions for the considered fractional nonlocal
multi-delay control system to be approximately controllable.
An example to illustrate the abstract results is given.

\bigskip


\bigskip

\end{abstract}




\section{Introduction}

We are concerned with the fractional delay control system
\begin{equation}
\label{eq:1.1}
^{C}D^{\alpha}_{t}u(t)+Au(t)=F(t, W_{\delta}(t))+ V_{\sigma}(t)
\end{equation}
subject to the fractional nonlocal condition
\begin{equation}
\label{eq:1.2}
^{L}D^{1-\alpha}_{t}[u(0)-u_{0}]=h[u(t)],
\end{equation}
where the unknown $u(\cdot)$ takes its values in a Hilbert space $H$
with norm $\Vert\cdot\Vert$, $^{C}D^{\alpha}_{t}$ and
$^{L}D^{1-\alpha}_{t}$ are the Caputo and Riemann-Liouville
fractional derivatives with $0<\alpha\leq1$ and $t\in
J=[0,a]$, respectively. Let $-A$ be a closed linear operator defined on a dense
set $S$ that generates a C$_{0}-$semigroup $Q(t), ~t\geq0$, of
bounded operators on $H$, $u_{0}\in S$. We assume that
$W_{\delta}=(A_{1}u_{\delta_{1}},\ldots,A_{p}u_{\delta_{p}})$ and
$V_{\sigma}=(B_{1}\mu_{\sigma_{1}}+\cdots+B_{q}\mu_{\sigma_{q}})$
are such that $\lbrace A_{i}(t):i=1,\ldots,p,~t\in J\rbrace$ is a family
of closed linear operators defined on dense sets
$S_{1},\ldots,S_{p}\supset S$ from $H$ into $H$ and
$\lbrace B_{j}(t): U\rightarrow H,~ j=1,\ldots,q,~t\in J\rbrace$ is
a family of bounded linear operators. The control function $\mu$
belongs to the space $L^{2}(J, U)$, a Hilbert space of admissible
control functions with $U$ as a Hilbert space, $\delta_{i},
\sigma_{j}: J\rightarrow J^{\prime}$ are delay arguments,
$i=1,\ldots,r$, $j=1,\ldots,s$, and $J^{\prime}=[0, t]$. The operators
$F: J\times H^{p}\rightarrow H$ and $h: C(J: H)\rightarrow H$ are
given abstract functions.

During the last decades, fractional differential equations have attract
the attention of many mathematicians, physicists and engineers
--- see, e.g., \cite{AMA.1,AMA.6,AMA.10,AMA.14,AMA.18,AMA.30}.
The reason is that real phenomena, such as dielectric and
electrode-electrolyte polarization, electromagnetic waves,
earthquakes, fluid dynamics, traffic, viscoelasticity and viscoplasticity,
can be described successfully and more accurately
using fractional models ---  see \cite{AMA.16,AMA.19,AMA.20,AMA.23,AMA.26}.
Fractional evolution equations with nonlocal conditions have
been studied in many works --- see \cite{AMA.5,AMA.7,AMA.21,AMA.33}
and references therein. Existence results to evolution equations
with nonlocal conditions in a Banach space
were first studied in \cite{AMA.3,AMA.4}.
In \cite{AMA.12} it is shown that, using the nonlocal condition
$u(0)+h(u)=u_{0}$ to describe, for instance, the diffusion
phenomenon of a small amount of gas in a transparent tube, can give
better results than using the standard local Cauchy condition
$u(0) = u_{0}$. According to \cite{AMA.12}, function $h$ takes
the form $h(u)=\sum_{k=1}^{p}c_{k}u(t_{k})$, where $c_{k}$,
$k = 1,\ldots,p$, are given constants and
$0 \leq t_{1}<\cdots<t_{p}\leq a$.
In \cite{MR2931380} the controllability of semilinear fractional differential
equations with nonlocal conditions in a Banach space
is investigated through the M\"{o}nch fixed point theorem, where
the semigroup generated by the linear part is not necessarily compact
but the nonlinear term satisfies some weak compactness condition.
The existence of mild solutions for
impulsive fractional evolution Cauchy problems
involving Caputo fractional derivatives is discussed
in \cite{MR2901608} by means of the theory of operators semigroup
and probability density functions via impulsive conditions,
while the solvability and optimal control of a class of fractional integrodifferential
evolution systems with an infinite delay in Banach spaces is studied in \cite{MR2872510}:
a concept for solution is introduced, existence and continuous dependence of solutions are investigated,
and existence of optimal controls proved. In \cite{MR2910321}
optimal feedback control laws for Lagrange problems
subject to semilinear fractional-order systems
in Banach spaces are established.

Exact controllability of fractional order systems has been proved
by many authors: \cite{AMA.2,AMA.8,AMA.9,AMA.28,AMA.29}. The main
tool is to convert the controllability question into a fixed point
problem with the assumption that the controllability operator has an
induced inverse on a quotient space. To prove controllability,
an assumption that the semigroup (resp. resolvent operator)
associated with the linear part is compact is then often made. However,
if the compactness condition holds on the bounded operator that maps
the control function or the generated $C_0$-semigroup, then the
controllability operator is also compact and its inverse
does not exist if the state space is infinite dimensional --- see
\cite{AMA.28}. Thus, the concept of exact controllability is too
strong in infinite dimensional spaces and the approximate
controllability notion is more appropriate.

Approximate controllability of integer order systems has been proved
in several works. In contrast, papers dealing with the approximate
controllability of fractional order systems are scarce. Recently,
the subject was addressed  in \cite{AMA.25,AMA.24},
while sufficient conditions for the (delay) approximate controllability
of fractional order systems, in which the nonlinear term depends
on both state and control variables, are investigated in \cite{AMA.17,AMA.27},
and the case of partial neutral fractional functional
differential systems with a state-dependent delay
is considered in \cite{AMA.31}.

Our main objective is to study the approximate
controllability of semilinear fractional control systems, where the
control function depends on multi-delay arguments and where the nonlocal
condition is fractional. The result is obtained under the assumption
that the associated linear system is approximately controllable. In
particular, the controllability question is transformed to a fixed
point problem for an appropriate nonlinear operator in a function
space. For that we need to construct a suitable set of
sufficient conditions.
The paper is organized as follows: in
Section~\ref{sec:2}, we present some essential definitions of
fractional calculus and basic facts in the semigroup theory that
will be used to obtain our main results. In Section~\ref{sec:3},
we state and prove existence and approximate
controllability results for problem \eqref{eq:1.1}--\eqref{eq:1.2}.
Finally, in Section~\ref{sec:4} we illustrate the new
results of the paper with an example.


\section{Preliminaries}
\label{sec:2}

In this section, we introduce some basic definitions, notations and lemmas,
which will be used throughout the work. In particular, we give necessary properties
of fractional calculus (see \cite{AMA.16,AMA.23,AMA.26}) and some fundamental facts
in semigroup theory (see \cite{AMA.15,AMA.22,AMA.32}).

Let $(H, \Vert\cdot\Vert)$ be a Hilbert space, $C(J, H)$ denote the Hilbert space
of continuous functions from $J$ into $H$ with the norm
$\Vert u\Vert_{J}=\sup \lbrace\Vert u(t)\Vert: t\in J\rbrace$, and $L(H)$
be the Hilbert space of bounded linear operators from $H$ to $H$. Further,
let $E(H)$ be the space of all bounded linear operators from $H$ to $H$
with the norm $\Vert G\Vert_{E(H)}=\sup\lbrace\Vert G(u)\Vert: \Vert u\Vert=1\rbrace$,
where $ G\in E(H)$ and $u\in H$. Throughout the paper, let $-A$ be the infinitesimal
generator of the C$_{0}-$semigroup ${Q(t)}$, $t\geq 0$, of uniformly bounded linear
operators on $H$. Clearly, $M=\sup_{t\in [0, \infty)}\Vert Q(t)\Vert<\infty$.

\begin{definition}
\label{def:2.1}
The fractional integral of order $\alpha>0$
of a function $f\in C([0, \infty))$ is given by
$$
I^{\alpha}f(t) := \frac{1}{\Gamma(\alpha)}
\int_{0}^{t}\frac{f(s)}{(t-s)^{1-\alpha}}ds,
\quad t>0,
$$
where $\Gamma$ is the gamma function, provided the right-hand side
is point-wise defined on $[0, \infty)$.
\end{definition}

\begin{definition}
\label{def:2.2}
The Riemann--Liouville derivative of order $\alpha > 0$
of a function $f\in C([0, \infty))$ is given by
$$
^{L}D^{\alpha}f(t) := \frac{1}{\Gamma(n-\alpha)}\frac{d^{n}}{dt^{n}}
\int_{0}^{t}\frac{f(s)}{(t-s)^{\alpha+1-n}}ds,
\quad t>0,
$$
where $n\in \mathbb{N}$ is such that $n-1<\alpha<n$.
\end{definition}

\begin{definition}
\label{def:2.3}
The Caputo derivative of order $\alpha > 0$
of a function $f\in C([0,\infty))$ is given by
$$
^{C}D^{\alpha}f(t) := {^{L}D^{\alpha}}\left(f(t)
-\sum\limits_{k=0}^{n-1}\frac{t^{k}}{k!}f^{(k)}(0)\right),
\quad t>0,
$$
where $n\in \mathbb{N}$ is such that $n-1<\alpha<n$.
\end{definition}

\begin{remark}
\label{rem:2.1}
The following properties hold:
\begin{enumerate}
\item If $f\in C^{n}([0, \infty))$, then
$$
^{C}D^{\alpha}f(t) = \frac{1}{\Gamma(n-\alpha)}
\int_{0}^{t}\frac{f^{(n)}(s)}{(t-s)^{\alpha+1-n}}ds
=I^{n-\alpha}f^{n}(t), \quad t>0, \quad n-1<\alpha<n.
$$

\item The Caputo derivative of a constant is equal to zero.

\item If $f$ is an abstract function with values in $H$,
then the integrals which appear in Definitions~\ref{def:2.1}--\ref{def:2.3}
are taken in Bochner's sense.
\end{enumerate}
\end{remark}

According to previous definitions, it is suitable to rewrite the
problem \eqref{eq:1.1}--\eqref{eq:1.2} in the equivalent integral form
\begin{equation}
\label{eq:2.1}
u(t)=u(0)+\frac{1}{\Gamma(\alpha)}\int_{0}^{t}(t-s)^{\alpha-1}
[-Au(s)+F(s, W_{\delta}(s))+ V_{\sigma}(s)]ds.
\end{equation}

\begin{remark}
\label{rem:2.2}
We note that:
\begin{enumerate}
\item For the nonlocal condition, the function $u(0)$ is dependent on $t$.

\item The Riemann--Liouville fractional derivative of $u(0)-u_{0}$
is well defined and $^{L}D^{1-\alpha}_{t}u_{0}\neq0$.

\item The function $u(0)$ takes the form $u_{0}+v_{0}+\frac{1}{\Gamma(1-\alpha)}
\int_{0}^{t}\frac{h[u(s)]}{(t-s)^{\alpha}}ds$, where $u(0)-u_{0}|_{t=0}=v_{0}$.

\item The explicit and implicit integrals given in \eqref{eq:2.1}
exist (taken in Bochner's sense).
\end{enumerate}
\end{remark}

\begin{definition}
\label{def:2.4}
A state $u \in C(J, H)$ is a mild solution of \eqref{eq:1.1}--\eqref{eq:1.2}
if, for each control $\mu\in L^{2}(J, U)$,
it satisfies the following integral equation:
$$
u(t)=S_{\alpha}(t)\left[u_{0}+v_{0}+\frac{1}{\Gamma(1-\alpha)}
\int_{0}^{t}\frac{h[u(s)]}{(t-s)^{\alpha}}ds\right]
+\int_{0}^{t}(t-s)^{\alpha-1}T_{\alpha}(t-s)
[F(s, W_{\delta}(s))+ V_{\sigma}(s)] ds,
$$
where
$$
S_{\alpha}(t)=\int_{0}^{\infty}\zeta_{\alpha}(\theta)Q(t^{\alpha}\theta)d\theta,
\quad T_{\alpha}(t)=\alpha\int_{0}^{\infty}\theta\zeta_{\alpha}(\theta)Q(t^{\alpha}\theta)d\theta,
$$
$$
\zeta_{\alpha}(\theta)=\frac{1}{\alpha}\theta^{-1-\frac{1}{\alpha}}
\varpi_{\alpha}(\theta^{-\frac{1}{\alpha}}) \geq 0,
\quad \varpi_{\alpha}(\theta)=\frac{1}{\pi}\sum_{n=1}^{\infty}
(-1)^{n-1}\theta^{-\alpha n-1}\frac{\Gamma(n\alpha+1)}{n!}
\sin(n\pi\alpha), \theta\in (0, \infty).
$$
\end{definition}

\begin{remark}
In Definition~\ref{def:2.4}, $\zeta_{\alpha}$
is a probability density function defined on $(0, \infty)$,
that is, $\zeta_{\alpha}(\theta)\geq 0$, $\theta\in(0, \infty)$, and
$\int_{0}^{\infty}\zeta_{\alpha}(\theta)d\theta=1$
(compare with \cite{AMA.11,AMA.13}; see also \cite{AMA.33,AMA.34}).
\end{remark}

The following lemma can be found in \cite{AMA.11,AMA.34}.

\begin{lemma}
\label{lem:2.1}
The operators $S_{\alpha}(t)$ and $T_{\alpha}(t)$
have the following properties:
\begin{enumerate}
\item
\label{lem:2.1:a}
For any fixed $t\geq0$, the operators $S_{\alpha}(t)$
and $T_{\alpha}(t)$ are linear and bounded, i.e.,
for any $u\in H$, $\Vert S_{\alpha}(t)u\Vert\leq M\Vert u\Vert$ and
$\Vert T_{\alpha}(t)u\Vert\leq \frac{M\alpha}{\Gamma (1+\alpha)}\Vert u\Vert$.

\item
\label{lem:2.1:b}
$\lbrace S_{\alpha}(t), t\geq0\rbrace$
and $\lbrace T_{\alpha}(t), t\geq0\rbrace$ are strongly continuous,
i.e., for $u\in H$ and $0\leq t_{1}<t_{2}\leq a$, one has
$\Vert S_{\alpha}(t_{2})u-S_{\alpha}(t_{1})u\Vert\rightarrow 0$
and $\Vert T_{\alpha}(t_{2})u-T_{\alpha}(t_{1})u\Vert\rightarrow 0$
as $t_{1}\rightarrow t_{2}$.
\end{enumerate}
\end{lemma}

Motivated by the recent works of \cite{AMA.8,AMA.17,AMA.24,AMA.25,AMA.31},
we make use of the following notions and lemmas.

Let $u_{a}(u(0); \mu)$ be the state value of \eqref{eq:1.1}--\eqref{eq:1.2}
at terminal time $a$, corresponding to the control $\mu$ and the nonlocal
value $u(0)$. For every $u_{0}, v_{0}\in H$, we introduce the set
$$
\mathfrak{R}(a, u(0)) := \left\lbrace
u_{a}\left(u_{0}+v_{0}+\frac{1}{\Gamma(1-\alpha)}
\int_{0}^{t}\frac{h[u(s)]}{(t-s)^{\alpha}}ds;
\mu\right)(0) : \mu(\cdot)\in L^{2}(J, U)\right\rbrace,
$$
which is called the reachable set associated with \eqref{eq:1.1}--\eqref{eq:1.2}
at terminal time $a$. Its closure in $H$ is denoted by
$\overline{\mathfrak{R}(a, u(0))}$.

\begin{definition}
\label{def:2.5}
The system \eqref{eq:1.1}--\eqref{eq:1.2} is said
to be approximately controllable on $J$ if
$\overline{\mathfrak{R}(a, u(0))}=H$, that is,
given an arbitrary $\epsilon>0$, it is possible
to steer in time $a$ the system from point $u(0)$ to all
points in the state space $H$ within a distance $\epsilon$.
\end{definition}

Consider the following linear nonlocal fractional multi-delay control system:
\begin{equation}
\label{eq:2.3}
^{C}D^{\alpha}_{t}u(t)+Au(t)
= B_{1}\mu(\sigma_{1}(t))+\cdots+B_{q}\mu(\sigma_{q}(t)),
\end{equation}
\begin{equation}
\label{eq:2.4}
^{L}D^{1-\alpha}_{t}[u(0)-u_{0}]=h[u(t)].
\end{equation}
The approximate controllability for the linear fractional nonlocal
multi-delay control system \eqref{eq:2.3}--\eqref{eq:2.4} is a
natural generalization of the notion of approximate controllability
of a linear first-order control system ($\alpha=1$, $\sigma_{j}(t)=t$,
$j=1$ and $h=0$). It is convenient at this point to introduce the
multi-delay controllability operator associated with
\eqref{eq:2.3}--\eqref{eq:2.4}. One has
\begin{equation*}
\Gamma^{a}_{0, \sigma_{j}}
= \int_{0}^{a}(a-s)^{\alpha-1}T_{\alpha}(a-s)B_{j}B_{j}^{\ast}T_{\alpha}^{\ast}(a-s)ds,
\quad j = 1, \ldots, q,
\end{equation*}
where $B_{j}^{\ast}$ denotes the adjoint of $B_{j}$
and $T_{\alpha}^{\ast}(t)$ is the adjoint
of $T_{\alpha}(t)$. We introduce here the new operator
$\Gamma^{a}_{0,\sigma} = \Gamma^{a}_{0, \sigma_{1}, \ldots,\sigma_{q}}$
of multi-delay controllability associated with \eqref{eq:2.3}--\eqref{eq:2.4} as
$$
\Gamma^{a}_{0, \sigma} := \int_{0}^{a}(a-s)^{\alpha-1}
T_{\alpha}(a-s)[B_{1}B_{1}^{\ast}+\cdots
+B_{q-1}B_{q-1}^{\ast}+B_{q}B_{q}^{\ast}]T_{\alpha}^{\ast}(a-s)ds.
$$

It is straightforward to see that $\Gamma^{a}_{0, \sigma}$
is a linear bounded positive operator. The following lemma
is proved in \cite{MR1720139} for linear positive operators
in Hilbert spaces, while a Banach space version is given in
\cite{MR2046377}. See also \cite{AMA.25,AMA.24}.

\begin{lemma}
\label{lem:2.2}
Let $\mathcal{R}(\beta, \Gamma^{a}_{0, \sigma})
=(\beta I+\Gamma^{a}_{0, \sigma})^{-1}$ for $\beta>0$.
The linear fractional control system \eqref{eq:2.3}--\eqref{eq:2.4}
is approximately controllable on $J$ if and only if
$\beta\mathcal{R}(\beta, \Gamma^{a}_{0, \sigma})\rightarrow 0$
as $\beta\rightarrow 0^{+}$ in the strong operator topology.
\end{lemma}

In the sequel we use Schauder's fixed point theorem that
can be found in any standard textbook of Functional Analysis
(see, e.g., \cite[Theorem 3.2, p. 119]{MR1987179}
or \cite[Theorem 4.1.1, p. 25]{MR0467717}):

\begin{lemma}
\label{lem:2.3}
If $\Omega$ is a closed bounded and convex
subset of a Banach space $X$ and $\psi: \Omega\rightarrow \Omega$ is
completely continuous, then $\psi$ has a fixed point in $\Omega$.
\end{lemma}


\section{Main results}
\label{sec:3}

We now formulate and establish our results
on the approximate controllability of the nonlocal delay system
\eqref{eq:1.1}--\eqref{eq:1.2}. With this purpose, firstly we prove
the existence of solutions for the fractional control system
\eqref{eq:1.1}--\eqref{eq:1.2} by using Schauder's fixed point theorem
(Lemma~\ref{lem:2.3}). Then, we show that under certain assumptions,
the approximate controllability of the fractional system
\eqref{eq:1.1}--\eqref{eq:1.2} is implied by the approximate controllability
of the corresponding linear system \eqref{eq:2.3}--\eqref{eq:2.4}.
Before starting, we make the following hypotheses:
\begin{enumerate}
\item[(H$_{1}$)] The semigroup $Q(t)$ is a compact operator for $t>0$.

\item[(H$_{2}$)] For each $t\in J$, the function
$F(t, \cdot): S_{1}\times\cdots\times S_{p}\rightarrow H$ is continuous and for each
$W_{\delta}\in C[(J^{\prime}, S_{1})\times\cdots\times (J^{\prime}, S_{p}); H]$,
in particular, for every element $u\in \cap_{i}S_{i}, i=1,\ldots,p$, the function
$F(\cdot, W_{\delta}): J\rightarrow H$ is strongly measurable.

\item[(H$_{3}$)] There exist functions $m_{i}\in L^{\frac{1}{1-\alpha}}(J^{\prime}, \mathbb{R}^{+})$
such that $\vert F(t, W_{\delta}(t))\vert\leq m_{1}(\delta_{1}(t))+\cdots+m_{p}(\delta_{p}(t))$
for all $u\in \cap_{i}S_{i}$, $0<\alpha<1$, $i=1,\ldots,p$, and almost all $t\in J^{\prime}$.

\item[(H$_{4}$)] The function $h: C(J: H)\rightarrow H$ is bounded in $H$, that is,
there exists a constant $k_{1}>0$ such that $\Vert h(u)\Vert_{H}\leq k_{1}$.

\item[(H$_{5}$)] The delay arguments $\delta_{i}, \sigma_{j}: J\rightarrow J^{\prime}$
are absolutely continuous and satisfy
$\vert\delta_{i}(t)\vert\leq t$ and $\vert\sigma_{j}(t)\vert\leq t$,
for every $t\in J$, $i=1,\ldots,r$ and $j=1,\ldots,s$.

\item[(H$_{6}$)] The function $F: J\times H^{p}\rightarrow H$ is continuous
and uniformly bounded and there exist $N_{\delta_{1}},\ldots,N_{\delta_{p}}>0$
such that $\Vert F(t, W_{\delta}(t))\Vert\leq N_{\delta_{1}}+\cdots+N_{\delta_{p}}$
for all $(t, W_{\delta})\in J\times H^{p}$.
\end{enumerate}

For the proof of our Theorem~\ref{thm:3.2}, we make use of Lemma~\ref{lem:3.1}
whose proof can be found in \cite{AMA.34}.

\begin{lemma}
\label{lem:3.1}
If the assumption (H$_{1}$) is satisfied, then $S_{\alpha}(t)$
and $T_{\alpha}(t)$ are also compact operators for every $t>0$.
\end{lemma}

\begin{theorem}
\label{thm:3.2}
If the hypotheses (H$_1$)--(H$_5$) are satisfied, then the fractional
nonlocal semilinear delay control system \eqref{eq:1.1}--\eqref{eq:1.2}
has a mild solution on $J$; here $M_{i}=\Vert m_{i}\Vert_{L^{\frac{1}{1-\alpha}}(J^{\prime})}$,
$i=1,\ldots,p$, and $M_{B_{j}}=\Vert B_{j}\Vert$, $j=1,\ldots,q$.
\end{theorem}

\begin{proof}
Consider the set
$$
S_{r} := \left\lbrace u\in C(J, H)\vert u(0)
=u_{0}+v_{0}+\frac{1}{\Gamma(1-\alpha)}
\int_{0}^{t}h[u(s)](t-s)^{-\alpha}ds, \Vert u\Vert\leq r\right\rbrace,
$$
where $r$ is a positive constant. For $\beta>0$,
we define the operator $\psi_{\beta}$ on $C(J, X)$ as follows:
$(\psi_{\beta}u)(t) := z(t)$ with
$$
z(t) := S_{\alpha}(t)u(0)+\int_{0}^{t}(t-s)^{\alpha-1}T_{\alpha}(t-s)[F(s,
W_{\delta}(s))+ B_{1}v(\sigma_{1}(s))+\cdots+B_{q}v(\sigma_{q}(s))]
ds,
$$
$v(\sigma_{1}(\cdot)) := B_{1}^{\ast}$, \ldots,
$v(\sigma_{q-1}(\cdot)) := B_{q-1}^{\ast}$,
$v(\sigma_{q}(t)) := B_{q}^{\ast}T_{\alpha}^{\ast}(a-t)
\mathcal{R}(\beta, \Gamma^{a}_{0, \sigma})p(u(\cdot))$, and
$$
p(u(\cdot)) := u_a-S_{\alpha}(a)u(0)
-\int_{0}^{a}(a-s)^{\alpha-1}T_{\alpha}(a-s)F(s, W_{\delta}(s))ds.
$$
In order to show that for all $\beta>0$ the operator $\psi_{\beta}$
from $C(J, H)$ into itself has a fixed point,
we divide the proof into several steps.
\textit{Step 1.} For $\beta>0$, there is a positive constant $r_{0}=r(\beta)$
such that $\psi_{\beta}: S_{r_{0}}\rightarrow S_{r_{0}}$. For any positive constant
$r$ and $u\in S_{r}$, and since $W_{\delta}(t)$ is continuous in $t$,
according to assumption (H$_2$) $F(t, W_{\delta}(t))$ is a measurable function on $J$
as well as function $(t-s)^{\alpha -1}\in L^{\frac{1}{\alpha}}(J^{\prime})$.
Using \ref{lem:2.1:a} of Lemma~\ref{lem:2.1}, (H$_3$) and H\"older's inequality, we get:
\begin{equation*}
\begin{split}
\Vert z(t)\Vert
&\leq M\Vert u(0)\Vert+\frac{M\alpha}{\Gamma (1+\alpha)}
\int_{0}^{t}(t-s)^{\alpha-1}\Vert F(s, W_{\delta}(s))
+ B_{1}v(\sigma_{1}(s))+\cdots+B_{q}v(\sigma_{q}(s))\Vert ds\\
&\leq M\left[\Vert u_{0}\Vert+\Vert v_{0}\Vert+\frac{k_{1}a^{1-\alpha}}{\Gamma(2-\alpha)}\right]\\
& \qquad +\frac{\alpha M \alpha^{\alpha}a^{2\alpha-1}}{\Gamma (1+\alpha)(2\alpha-1)^{\alpha}}
\left[M_{1}+\cdots+M_{p}+ M^{2}_{B_{1}}+\cdots+M^{2}_{B_{q-1}}+M_{B_{q}}\Vert v(\sigma_{q}(s))\Vert\right]
\end{split}
\end{equation*}
and
$$
\Vert v(\sigma_{q}(t))\Vert=\frac{1}{\beta}M_{B_{q}}M\left[\Vert u_{a}\Vert
+M\Vert u(0)\Vert+\frac{\alpha M \alpha^{\alpha}a^{2\alpha-1}\left(
M_{1}+\cdots+M_{p}\right)}{\Gamma (1+\alpha)(2\alpha-1)^{\alpha}}\right].
$$
We deduce that for large enough $r_{0}>0$, the inequality $\Vert(\psi_{\beta}u)(t)\Vert\leq r_{0}$
holds, i.e., $(\psi_{\beta}u)\in S_{r_{0}}$. Therefore, $\psi_{\beta}$ maps $S_{r_{0}}$ into itself.
\textit{Step 2.} For each $0<\alpha\leq1$, the operator $\psi_{\beta}$ maps $S_{r_{0}}$
into a relatively compact subset of $S_{r_{0}}$. Based on the infinite-dimensional
version of the Ascoli-Arzela theorem, we show that:
(i) the set $V(t) := \lbrace(\psi_{\beta}u)(t):
u(\cdot)\in S_{r_{0}}\rbrace$ is relatively compact in $H$
for any $t\in J$; (ii) the family of functions
$\lbrace(\psi_{\beta}u), u\in S_{r_{0}}\rbrace$
is relatively compact (for this, it suffices
to prove that $V(t)$ is bounded and equicontinuous).
We begin by proving (i). Let $t$ be a fixed real number, and let $\tau$
be a given real number satisfying $0\leq \tau<t$. For any $\eta>0$, define
\begin{equation*}
\begin{split}
(\psi^{\tau, \eta}_{\beta}u)(t)
&=\int_{\eta}^{\infty}\zeta_{\alpha}(\theta)Q(t^{\alpha}\theta)\left[u_{0}
+v_{0}+\frac{1}{\Gamma(1-\alpha)}\int_{0}^{t-\tau}(t-s)^{-\alpha}h(u(s))ds\right]d\theta\\
&\qquad +\alpha\int_{0}^{t-\tau}\int_{\eta}^{\infty}(t-s)^{\alpha-1}\theta\zeta_{\alpha}(\theta)
Q\left((t-s)^{\alpha}\theta\right)F(s, W_{\delta}(s))d\theta ds\\
& \qquad +\alpha\int_{0}^{t-\tau}\int_{\eta}^{\infty}(t-s)^{\alpha-1}
\theta\zeta_{\alpha}(\theta)Q((t-s)^{\alpha}\theta)V_{\sigma}(s)d\theta ds\\
&= Q(\tau^{\alpha}\eta)\int_{\eta}^{\infty}\zeta_{\alpha}(\theta)Q(t^{\alpha}\theta
-\tau^{\alpha}\eta)\left[u_{0}+v_{0}+\frac{1}{\Gamma(1-\alpha)}
\int_{0}^{t-\tau}(t-s)^{-\alpha}h(u(s))ds\right]d\theta\\
&\qquad + Q(\tau^{\alpha}\eta)\alpha\int_{0}^{t-\tau}\int_{\eta}^{\infty}(t-s)^{\alpha-1}
\theta\zeta_{\alpha}(\theta)Q((t-s)^{\alpha}\theta-\tau^{\alpha}\eta)F(s, W_{\delta}(s))d\theta ds\\
&\qquad + Q(\tau^{\alpha}\eta)\alpha\int_{0}^{t-\tau}\int_{\eta}^{\infty}(t-s)^{\alpha-1}
\theta\zeta_{\alpha}(\theta)Q((t-s)^{\alpha}\theta-\tau^{\alpha}\eta)V_{\sigma}(s)d\theta ds\\
& := Q(\tau^{\alpha}\eta)y(t, \tau).
\end{split}
\end{equation*}
Because $Q(\tau^{\alpha}\eta)$ is compact and $y(t, \tau)$ is bounded on $S_{r_{0}}$,
$\lbrace (\psi^{\tau, \eta}_{\beta}u)(t): u(\cdot)\in S_{r_{0}}\rbrace$
is a relatively compact set in $H$. On the other hand,
\begin{equation*}
\begin{split}
\Vert(\psi_{\beta}u)&(t)-(\psi^{\tau, \eta}_{\beta}u)(t)\Vert\\
&=\left\Vert\int_{0}^{\eta}\zeta_{\alpha}(\theta)Q(t^{\alpha}\theta)\left[u_{0}
+v_{0}+\frac{1}{\Gamma(1-\alpha)}\int_{0}^t(t-s)^{-\alpha}h(u(s))ds\right]d\theta\right.\\
&\qquad +\left.\int_{\eta}^{\infty}\zeta_{\alpha}(\theta)Q(t^{\alpha}\theta)\left[u_{0}
+v_{0}+\frac{1}{\Gamma(1-\alpha)}\int_{t-\tau}^{t}(t-s)^{-\alpha}h(u(s))ds\right]d\theta\right\Vert\\
&\qquad +\alpha\left\Vert\int_{0}^t\int_{0}^{\eta}(t-s)^{\alpha-1}\theta\zeta_{\alpha}(\theta)
Q((t-s)^{\alpha}\theta)[F(s, W_{\delta}(s))+V_{\sigma}(s)]d\theta ds\right.\\
&\qquad +\left.\int_{t-\tau}^{t}\int_{\eta}^{\infty}(t-s)^{\alpha-1}\theta\zeta_{\alpha}(\theta)
Q((t-s)^{\alpha}\theta)[F(s, W_{\delta}(s))+V_{\sigma}(s)]d\theta ds\right\Vert\\
&\leq M\left\lbrace\left[\Vert u_{0}\Vert+\Vert v_{0}\Vert
+\frac{k_{1}a^{1-\alpha}}{\Gamma(2-\alpha)}\right]
\int_{0}^{\eta}\zeta_{\alpha}(\theta)d\theta
+\left[\Vert u_{0}\Vert+\Vert v_{0}\Vert
+\frac{k_{1}\tau^{1-\alpha}}{\Gamma(2-\alpha)}\right]\right\rbrace\\
&\qquad +\frac{\alpha M \alpha^{\alpha}a^{2\alpha-1}}{\Gamma(1+\alpha)(2\alpha-1)^{\alpha}}\left\lbrace
M_{1}+\cdots+M_{p}+ M^{2}_{B_{1}}+\cdots+M^{2}_{B_{q-1}}\right.\\
&\qquad \left. +\frac{1}{\beta}M^{2}_{B_{q}}M
\left[\Vert u_{a}\Vert+M\Vert u(0)\Vert+\frac{\alpha
M \alpha^{\alpha}a^{2\alpha-1}(M_{1}+\cdots+M_{p})}{\Gamma (1+\alpha)(2\alpha-1)^{\alpha}}\right]\right\rbrace
\int_{0}^{\eta}\theta\zeta_{\alpha}(\theta)d\theta\\
&\qquad+ \frac{\alpha M \alpha^{\alpha}\tau^{2\alpha-1}}{\Gamma (1+\alpha)(2\alpha-1)^{\alpha}}\left\lbrace
M_{1}+\cdots+M_{p}+ M^{2}_{B_{1}}+\cdots+M^{2}_{B_{q-1}}\right.\\
&\qquad \left. +\frac{1}{\beta}M^{2}_{B_{q}}M \left[\Vert u_{a}\Vert+M\Vert u(0)\Vert+\frac{\alpha M
\alpha^{\alpha}a^{2\alpha-1}(M_{1}+\cdots+M_{p})}{\Gamma (1+\alpha)(2\alpha-1)^{\alpha}}\right]\right\rbrace.
\end{split}
\end{equation*}
This implies that there are relatively compact sets arbitrarily close to $V(t)$
for each $t\in(0, a]$. Hence, $V(t)$, $t\in(0, a]$, is relatively compact in $H$.
We now prove (ii). First we show that
$V(t):=\lbrace(\psi_{\beta}u)(\cdot): u(\cdot)\in S_{r_{0}}\rbrace$
is an equicontinuous family of functions on $[0, a]$. For any $u\in S_{r_{0}}$
and $0\leq t_{1}\leq t_{2}\leq a$,
\begin{equation*}
\begin{split}
& \Vert z(t_{2})-z(t_{1})\Vert
\leq \left\Vert S_{\alpha}(t_{2})[u_{0}
+v_{0}+\frac{1}{\Gamma(1-\alpha)}\int_{t_{1}}^{t_{2}}(t_{2}-s)^{-\alpha}h(u(s))ds]\right\Vert\\
&\quad + \left\Vert S_{\alpha}(t_{2})[u_{0}+v_{0}+\frac{1}{\Gamma(1-\alpha)}
\int_{0}^{t_{1}}[(t_{2}-s)^{-\alpha}-(t_{1}-s)^{-\alpha}]h(u(s))ds]\right\Vert\\
&\quad + \left\Vert[S_{\alpha}(t_{2})-S_{\alpha}(t_{1})][u_{0}+v_{0}
+\frac{1}{\Gamma(1-\alpha)}\int_{0}^{t_{1}}(t_{1}-s)^{-\alpha}h(u(s))ds]\right\Vert\\
&\quad + \left\Vert\int_{t_{1}}^{t_{2}}(t_{2}-s)^{\alpha-1}T_{\alpha}(t_{2}-s)
\left[F(s, W_{\delta}(s))
+ B_{1}v(\sigma_{1}(s))+\cdots+B_{q}v(\sigma_{q}(s))\right]ds\right\Vert\\
&\quad + \left\Vert\int_{0}^{t_{1}}[(t_{2}-s)^{\alpha-1}
-(t_{1}-s)^{\alpha-1}]T_{\alpha}(t_{2}-s)\left[F(s, W_{\delta}(s))
+ B_{1}v(\sigma_{1}(s))+\cdots+B_{q}v(\sigma_{q}(s))\right]ds\right\Vert\\
&\quad + \left\Vert\int_{0}^{t_{1}}(t_{1}-s)^{\alpha-1}[T_{\alpha}(t_{2}-s)-T_{\alpha}(t_{1}-s)]
\left[F(s, W_{\delta}(s))+ B_{1}v(\sigma_{1}(s))
+\cdots+B_{q}v(\sigma_{q}(s))\right]ds\right\Vert\\
&\leq I_{1}+I_{2}+I_{3}+I^{*}_{1}+I^{*}_{2}+I^{*}_{3}.
\end{split}
\end{equation*}
We have
\begin{gather*}
I_{1}\leq M\left\lbrace\Vert u_{0}\Vert+\Vert v_{0}\Vert
+\frac{k_{1}(t_{2}-t_{1})^{1-\alpha}}{\Gamma(2-\alpha)}\right\rbrace,\\
I_{2}\leq M\left\lbrace\Vert u_{0}\Vert+\Vert v_{0}\Vert+\frac{k_{1}[(t_{2}
-t_{1})^{1-\alpha}+t_{2}^{1-\alpha}+t_{1}^{1-\alpha}]}{\Gamma(2-\alpha)}\right\rbrace,\\
I_{3}\leq 2M\left\lbrace\Vert u_{0}\Vert+\Vert v_{0}\Vert
+\frac{k_{1}t_{1}^{1-\alpha}}{\Gamma(2-\alpha)}\right\rbrace.
\end{gather*}
Using H\"older's inequality and assumption (H$_3$), one gets
\begin{gather*}
I^{*}_{1}\leq \frac{\alpha M \alpha^{\alpha}(t_{2}-t_{1})^{2\alpha-1}}{
\Gamma (1+\alpha)(2\alpha-1)^{\alpha}} \left[M_{1}+\cdots+M_{p}+ M^{2}_{B_{1}}
+\cdots+M^{2}_{B_{q-1}}+M_{B_{q}}\Vert v(\sigma_{q})\Vert\right],\\
I^{*}_{2}\leq \frac{\alpha M \alpha^{\alpha}(t_{2}-t_{1})^{2\alpha-1}}{\Gamma (1
+\alpha)(2\alpha-1)^{\alpha}} \left[M_{1}+\cdots+M_{p}+ M^{2}_{B_{1}}+\cdots
+M^{2}_{B_{q-1}}+M_{B_{q}}\Vert v(\sigma_{q})\Vert\right].
\end{gather*}
Obviously, $I^{*}_{3}=0$ for $t_{1}=0$ and $0<t_{2}\leq a$.
For $t_{1}>0$ and $\epsilon>0$ small enough, we obtain
\begin{equation*}
\begin{split}
I^{*}_{3} &\leq \int_{0}^{t_{1}-\epsilon}(t_{1}-s)^{\alpha-1}\Vert
T_{\alpha}(t_{2}-s)-T_{\alpha}(t_{1}-s)\Vert\\
&\qquad \times[\Vert F(s, W_{\delta}(s))\Vert
+ \Vert B_{1}v(\sigma_{1}(s))\Vert+\cdots+\Vert B_{q}v(\sigma_{q}(s))\Vert]ds\\
&\qquad+\int_{t_{1}-\epsilon}^{t_{1}}(t_{1}-s)^{\alpha-1}\Vert
T_{\alpha}(t_{2}-s)-T_{\alpha}(t_{1}-s)\Vert\\
&\qquad\times[\Vert F(s, W_{\delta}(s))\Vert
+ \Vert B_{1}v(\sigma_{1}(s))\Vert+\cdots+\Vert B_{q}v(\sigma_{q}(s))\Vert]ds
\end{split}
\end{equation*}
\begin{equation*}
\begin{split}
&\leq \frac{\alpha^{\alpha}[t_{1}^{\frac{2\alpha-1}{\alpha}}
-\epsilon^{\frac{2\alpha-1}{\alpha}}]^{\alpha}}{\Gamma (1+\alpha)(2\alpha-1)^{\alpha}}
\left[M_{1}+\cdots+M_{p}+ M^{2}_{B_{1}}+\cdots+M^{2}_{B_{q-1}}+M_{B_{q}}\Vert v(\sigma_{q})\Vert\right]\\
&\qquad \times\sup_{s\in [0, t_{1}-\epsilon]}\Vert T_{\alpha}(t_{2}-s)-T_{\alpha}(t_{1}-s)\Vert\\
&\qquad +\frac{2\alpha M\alpha^{\alpha}\epsilon^{2\alpha-1}}{\Gamma (1+\alpha)(2\alpha-1)^{\alpha}}
\left[M_{1}+\cdots+M_{p}+ M^{2}_{B_{1}}+\cdots+M^{2}_{B_{q-1}}+M_{B_{q}}\Vert v(\sigma_{q})\Vert\right].
\end{split}
\end{equation*}
Note that $I_{1}, I_{2}, I_{3}, I^{*}_{1}, I^{*}_{2}\rightarrow 0$ as $t_{2}-t_{1}\rightarrow 0$.
Moreover, the assumption (H$_1$) together with Lemma~\ref{lem:3.1} imply the continuity of
$T_{\alpha}(t)$ in $t$ in the uniform operator topology. It is easy to verify that
$I^{*}_{3}$ tends to zero independently of $u\in S_{r_{0}}$ as
$t_{2}-t_{1}\rightarrow 0, \epsilon\rightarrow 0$. Consequently,
$I_{1}+I_{2}+I_{3}+I^{*}_{1}+I^{*}_{2}+I^{*}_{3}$ does not depend
on the particular choices of $u(\cdot)$ and tends to zero as $t_{2}-t_{1}\rightarrow 0$,
which means that $\lbrace(\psi_{\beta}u), u\in S_{r_{0}}\rbrace$ is equicontinuous.
Therefore, $\psi_{\beta}[S_{r_{0}}]$ is equicontinuous and also bounded.
By the Ascoli--Arzela theorem, $\psi_{\beta}[S_{r_{0}}]$ is relatively compact
in $C(J, H)$. On the other hand, it is easy to see that for all $\beta>0$, $\psi_{\beta}$
is continuous on $C(J, H)$. Hence, for all $\beta>0$, $\psi_{\beta}$ is a completely
continuous operator on $C(J, H)$. According to Schauder's fixed point theorem
(Lemma~\ref{lem:2.3}), $\psi_{\beta}$ has a fixed point. Thus, the fractional nonlocal
control system \eqref{eq:1.1}--\eqref{eq:1.2} has a mild solution on $J$.
\end{proof}

\begin{theorem}
\label{thm:3.3}
Assume that (H$_{1}$)--(H$_{6}$) are satisfied and the linear system
\eqref{eq:2.3}--\eqref{eq:2.4} is approximately controllable on $J$.
Then the semilinear fractional nonlocal delay system
\eqref{eq:1.1}--\eqref{eq:1.2} is approximately controllable on $J$.
\end{theorem}

\begin{proof}
Let $\hat{u}_{\beta}(\cdot)$ be a fixed point of $\psi_{\beta}$ in $S_{r_{0}}$.
By Theorem~\ref{thm:3.2}, any fixed point of $\psi_{\beta}$ is a mild solution
of \eqref{eq:1.1}--\eqref{eq:1.2} under the multi-delay controls
$$
\left\{\begin{array}{lll}
\hat{\mu}_{\beta}(\sigma_{1}(\cdot))=B_{1}^{\ast},\\
\qquad \vdots\\
\hat{\mu}_{\beta}(\sigma_{q-1}(\cdot))=B_{q-1}^{\ast},\\
\hat{\mu}_{\beta}(\sigma_{q}(t))=B_{q}^{\ast}T_{\alpha}^{\ast}(a-t)
\mathcal{R}(\beta, \Gamma^{a}_{0, \sigma})p(\hat{u}_{\beta}),
\end{array} \right.
$$
and satisfies
\begin{equation}
\label{eq:3.1}
\hat{u}_{\beta}(a)=u_{a}+\beta\mathcal{R}(\beta,
\Gamma^{a}_{0, \sigma})p(\hat{u}_{\beta}).
\end{equation}
From condition (H$_{6}$), it follows that
$$
\int_{0}^{a}\Vert F(s, \hat{W}_{\beta}(\delta(s))\Vert^{2}ds
\leq a[N_{\delta_{1}}+\cdots+N_{\delta_{p}}]^{2},
$$
where $\hat{W}_{\beta}(\delta)=(A_{1}\hat{u}_{\beta}(\delta_{1}),
\ldots,A_{p}\hat{u}_{\beta}(\delta_{p}))$. Consequently, the sequence
$\lbrace F(t,\hat{W}_{\beta}(\delta(t))\rbrace_{t\in J}$ is bounded
in $L_{2}(J,H)$. Then there is a subsequence, denoted by
$\lbrace F_{k}(t,\hat{W}_{\beta}(\delta(t))\rbrace_{t\in J}$,
that converges weakly to a function $f(t)$ in $L_{2}(J, H)$. Define
$$
\omega := u_{a}-S_{\alpha}(a)u(0)-\int_{0}^{a}(a-s)^{\alpha-1}T_{\alpha}(a-s)f(s)ds.
$$
It follows that
\begin{equation}
\label{eq:3.2}
\begin{aligned}
\Vert p(\hat{u}_{\beta})-\omega\Vert&=\left\Vert
\int_{0}^{a}(a-s)^{\alpha-1}T_{\alpha}(a-s)[
F_{k}(t, \hat{W}_{\beta}(\delta(t))-f(s)]ds\right\Vert\\
&\leq\sup\limits_{t\in J}\left\Vert \int_{0}^{t}
(t-s)^{\alpha-1}T_{\alpha}(a-s)[F_{k}(t, \hat{W}_{\beta}(\delta(t))-f(s)]ds\right\Vert.
\end{aligned}
\end{equation}
Similarly as in the proof of Theorem~\ref{thm:3.2}, using the infinite-dimensional
version of the Ascoli--Arzela theorem one can show that the operator
$l(\cdot)\rightarrow \int_{0}^{\cdot}(\cdot-s)T_{\alpha}(\cdot-s)l(s)ds:
L_{2}(J, H)\rightarrow C(J, H)$ is compact. Consequently, the right-hand side
of the inequality \eqref{eq:3.2} tends to zero as $\beta\rightarrow 0^{+}$.
Then, from \eqref{eq:3.1}, we obtain that
\begin{equation}
\begin{aligned}
\Vert \hat{u}_{\beta}(a)-u_{a}\Vert
&=\Vert \beta\mathcal{R}(\beta, \Gamma^{a}_{0, \sigma})(\omega)\Vert
+\Vert \beta\mathcal{R}(\beta, \Gamma^{a}_{0, \sigma})\Vert \Vert p(\hat{u}_{\beta})-\omega\Vert\\
&\leq\Vert \beta\mathcal{R}(\beta, \Gamma^{a}_{0, \sigma})(\omega)\Vert
+\Vert p(\hat{u}_{\beta})-\omega\Vert\rightarrow 0.
\end{aligned} \label{eq:3.3}
\end{equation}
Using the inequality \eqref{eq:3.2} and Lemma~\ref{lem:2.2},
we conclude that \eqref{eq:3.3} tends to zero as $\beta\rightarrow 0^{+}$.
Hence, system \eqref{eq:1.1}--\eqref{eq:1.2} is approximately controllable on $J$.
\end{proof}


\section{An example}
\label{sec:4}

Consider the fractional nonlocal partial multi-delay control system
\begin{equation}
\label{eq:4.1}
\frac{\partial^{\alpha}u(x, t)}{\partial t^{\alpha}}+a(x)\frac{\partial^{2}u(x,t)}{\partial x^{2}}
=\Phi(t, D^{p}_{x}u(x, \delta_{p}(t)))+B^{q}\mu(x, \sigma_{q}(t))
\end{equation}
subject to
\begin{equation}
\label{eq:4.2}
u(x, 0)=g(x)+\sum\limits_{k=1}^{m}\frac{c_{k}}{\Gamma(1-\alpha)}
\int_{0}^{t_{k}}\frac{u(x, s_{k})}{(t_{k}-s_{k})^{\alpha}}ds_{k},
\quad x\in [0,\pi],
\end{equation}
\begin{equation}
\label{eq:4.3}
u(0,t)=u(\pi,t)=0,~t\in J,
\end{equation}
where $0<\alpha\leq1$, $0< t_{1}<\cdots<t_{m}< a$,
$\Phi$ is given as $F$, and the function $a(x)$ is continuous.
Let us take the operators $D^{p}_{x}$ and $B^{q}$ as follows:
\begin{gather*}
D^{p}_{x}u(x, \delta_{p}(t))=(\partial_{x} u(x, \sin t),
\partial_{x}^{2} u(x, \sin t/2),\ldots,\partial_{x}^{p} u(x, \sin t/p)),\\
B^{q}\mu(x, \sigma_{q}(t))=\xi(x, \sin t)+\xi(x, \sin t/2)+\cdots+\xi(x, \sin t/q),
\end{gather*}
such that the control function is $\mu(t)=\xi(\cdot, t)$, where
$\xi: [0,\pi]\times J\rightarrow [0,\pi]$ is continuous, the multi-delays
$\delta_{\tau}(t)=\sigma_{\tau}(t)=\sin(t/\tau)$,
$\tau=1$, $\ldots$, $\eta$, $\eta=\max(p, q)$, and the nonlocal function
is given by $h(u(\cdot, t))=\sum_{k=1}^{m}c_{k}u(\cdot, t_{k})$.
Assume that
$$
H=L^{2}[0, \pi],~ S_{r}
=\lbrace y\in L^{2}[0,\pi]: \Vert y\Vert\leq r\rbrace.
$$
We define $A: H\rightarrow H$ by $(Aw)(x)=a(x)w^{\prime\prime}$ with the domain
$$
D(A)=\lbrace w\in H: w, w^{\prime} \text{ are absolutely continuous, }
w^{\prime\prime}\in H, w(0)=w(\pi)=0\rbrace,
$$
dense in the Hilbert space $H$. Then,
$$
Aw=\sum\limits_{n=1}^{\infty}n^{2}(w,w_{n})w_{n}, \quad w\in D(A),
$$
where $(\cdot, \cdot)$ is the inner product in $L^{2}[0, \pi]$ and
$w_{n}(t)=\sqrt{\frac{2}{\pi}}\sin nt^{\alpha}$, $0<\alpha\leq 1$,
$t\in J$, $n=1,2,\ldots$, is the orthogonal set of eigenvectors in $A$.
It is well known that $A$ generates a compact, analytic and self-adjoint
semigroup $\lbrace Q(t)$, $t\geq 0\rbrace$ in $H$. For all $t\geq0$ and $w\in H$,
\begin{equation}
\label{eq:semigroupQ}
Q(t)w
=\sum\limits_{n=1}^{\infty}e^{-n^{2}t^{\alpha}}(w, w_{n})w_{n},
\quad \Vert Q(t)\Vert \leq e^{-t}.
\end{equation}
Therefore, the problem \eqref{eq:4.1}--\eqref{eq:4.3} is
a formulation of the control system \eqref{eq:1.1}--\eqref{eq:1.2}.
Moreover, all the assumptions (H$_1$)--(H$_6$) hold. Then, the
associated linear system of \eqref{eq:4.1}--\eqref{eq:4.3} is not
exactly controllable but it is approximately controllable.
Indeed, since the semigroup $Q(t)$ given by \eqref{eq:semigroupQ}
is compact, then the controllability operator
is also compact and hence the induced inverse does not exist because the state space is infinite dimensional.
Thus, the concept of exact controllability is too strong.
On the other hand, we have $\lambda\mathcal{R}(\lambda, \Gamma^{a}_{0,i})\rightarrow 0$ as $\lambda\rightarrow 0^{+}$,
$i=1,2$, in the strong operator topology, which is a necessary and sufficient condition
for the linear system to be approximately controllable.
Hence, by Theorems~\ref{thm:3.2} and \ref{thm:3.3}, the control system
\eqref{eq:4.1}--\eqref{eq:4.3} is approximately controllable on $J$.


\section*{Acknowledgments}

This work was partially supported by FEDER funds through
COMPETE -- Operational Programme Factors of Competitiveness
(``Programa Operacional Factores de Competitividade''),
and by Portuguese funds through the Center for Research and Development in Mathematics and Applications
(CIDMA, University of Aveiro) and the Portuguese Foundation for Science and Technology
(``FCT -- Funda\c{c}\~{a}o para a Ci\^{e}ncia e a Tecnologia''), within project PEst-C/MAT/UI4106/2011
with COMPETE number FCOMP-01-0124-FEDER-022690. The authors are grateful to
FCT and CIDMA for the post-doc fellowship BPD/UA/CIDMA/2011, PD2012-MTSC,
and to three anonymous referees for valuable suggestions and comments,
which improved the quality of the paper.



\label{lastpage}


\end{document}